\documentclass[11pt]{article}
\usepackage{amsthm, amsmath, amssymb, amsfonts, url, booktabs, tikz, setspace, fancyhdr, bm}
\usepackage[hidelinks]{hyperref}
\usepackage{geometry}
\usepackage{hyperref, enumerate}
\usetikzlibrary{decorations.pathmorphing,decorations.pathreplacing,decorations.shapes}

\newtheorem{theorem}{Theorem}[section]
\newtheorem{claim}[theorem]{Claim}

\newtheorem{proposition}[theorem]{Proposition}
\newtheorem*{remark*}{Remark}

\newcommand{\floor}[1]{\left\lfloor{#1}\right\rfloor}
\newcommand{\ceil}[1]{\left\lceil{#1}\right\rceil}
\DeclareMathOperator{\ex}{ex}

\newenvironment{poc}{\begin{proof}[Proof of claim]}{\end{proof}}

\def\C{\mathcal C}

\def\F{\mathcal F}

\usepackage{todonotes}

\newcommand*{\abs}[1]{\lvert#1\rvert}

\title{How connectivity affects the extremal number of trees}
\author{Suyun Jiang\thanks{School of Artificial Intelligence, Jianghan University, Wuhan, Hubei, China, and Extremal Combinatorics and Probability Group (ECOPRO), Institute for Basic Science (IBS), Daejeon, South Korea. Supported by National Natural Science Foundation of China (11901246) and China Scholarship Council and IBS-R029-C4.
Email: jiang.suyun@163.com. }
\and
Hong Liu\thanks{Extremal Combinatorics and Probability Group (ECOPRO), Institute for Basic Science (IBS), Daejeon, South Korea.
Supported by IBS-R029-C4.
 Emails: hongliu@ibs.re.kr. }
\and
Nika Salia\thanks{King Fahd University of Petroleum and Minerals, Dhahran, Saudi Arabia. Supported by  the National Research, Development and Innovation Office NKFIH grant K132696. Email: salianika@gmail.com.}
}

\begin{document}
\maketitle
\begin{abstract}
The Erd\H{o}s-Sós conjecture states that the maximum number of edges in an $n$-vertex graph without a given $k$-vertex tree is at most $\frac {n(k-2)}{2}$. 
Despite significant interest, the conjecture remains unsolved. Recently, Caro, Patkós, and Tuza considered this problem for host graphs that are connected. Settling a problem posed by them, for a $k$-vertex tree $T$, we construct $n$-vertex connected graphs that are $T$-free with at least $(1/4-o_k(1))nk$ edges, showing that the additional connectivity condition can reduce the maximum size by at most a factor of 2. Furthermore, we show that this is optimal: there is a family of $k$-vertex brooms $T$ such that the maximum size of an $n$-vertex connected $T$-free graph is at most $(1/4+o_k(1))nk$.
\end{abstract}

\section{Introduction}
In extremal graph theory, a central focus is determining the extremal number of various graphs. 
The extremal number, denoted by $\ex(n,F)$, is the maximum number of edges in an $n$-vertex graph that does not contain a graph $F$ as a subgraph, not necessarily induced. 
While the asymptotic behavior of this function has been determined for all non-bipartite graphs by Erd\H{o}s,  Stone, and Simonovits~\cite{erdos1946structure, erdHos1965limit}, the behavior for bipartite graphs remains open with significant interest from the community. 

The Erd\H{o}s-Gallai theorem~\cite{erdHos1959maximal}, established in 1959, studied the $k$-vertex path $P_k$, stating that $\ex(n,P_k)\le \frac{n(k-2)}{2}$, and the maximum value is achieved by $\cup K_{k-1}$, the disjoint union of $K_{k-1}$ when $k-1 \big| n$. Once the extremal number of the path is determined, extending the question to a tree is a natural next step. In 1962, Erd\H{o}s and Sós~\cite{erd1964extremal} conjectured that for any $k$-vertex tree, its extremal number is at most $\frac{n(k-2)}{2}$, and again the disjoint union of $K_{k-1}$ serves as an example for tightness. 

Motivated by the fact that the conjectured maximizer $\cup K_{k-1}$ is not connected, a natural variant is to consider host graphs that are connected, see e.g~\cite{furedi2013history, caro2022connected}. Formally, the connected extremal number $\ex_{c}(n,F)$ is the maximum number of edges in an $n$-vertex connected graph without a subgraph isomorphic to $F$. 
While the additional connectivity condition does not affect the asymptotics of the extremal number when the forbidden graph is non-bipartite or 2-edge-connected, Caro, Patk\'{o}s, and Tuza~\cite{caro2022connected} investigated what effect it has for trees. Notice that, in contrast with the classical extremal number, its connected relative is not even a  monotone function of $n$. Indeed, for paths, it is known that for every $k\geq 10$, $\ex_{c}(k,P_k)=\binom{k-2}{2}+2<\binom{k-1}{2}=\ex_{c}(k-1,P_k)$.

Such a connected variant for trees was in fact studied before and could date back to the work of Kopylov~\cite{kopylov1977maximal} in 1977, in which he resolved the problem for paths, showing that for $n\geq k \geq 4$,
\begin{equation}\label{Eq:Erd_gallai_connected}
  \ex_{c}(n,P_{k})=\max \left\{\binom{k-2}{2}+(n-(k-2)), \floor{\frac{k-2}{2}}\left(n-\ceil{\frac{k}{2}}\right)+\binom{\ceil{\frac{k}{2}}}{2} \right\}.
\end{equation}
Later, Balister, Gy{\H{o}}ri, Lehel, Schelp~\cite{balister2008connected} characterized extremal graphs for every $n$. There are also recent developments towards the stability version of this theorem by F{\"u}redi, Kostochka, Luo, Verstra{\"e}te~\cite{furedi2018stability,furedi2016stability}.

By the Erd\H{o}s-Gallai theorem and Kopylov's result~\eqref{Eq:Erd_gallai_connected}, we see that the asymptotic of the maximum number of edges in an $n$-vertex $P_k$-free graph does not change by an additional connectivity constraint as $n\rightarrow \infty$ and $k\rightarrow \infty$. Caro, Patk\'{o}s, and Tuza~\cite{caro2022connected} studied how much smaller $\ex_c(n,T)$ can be for a $k$-vertex tree $T$, compared to $\frac{n(k-2)}{2}$. Formally, they defined
\begin{equation}\label{Eq:Def_gamma}
\gamma_{k}:=\inf\Big\{\limsup_{n \to \infty}\frac{\ex_c(n,T_k)}{\frac{n(k-2)}{2}}:T_k \mbox{~is~a~$k$-vertex~tree}\Big\}\mbox{~ and ~} \gamma:=\lim_{k \to \infty}\gamma_k.
\end{equation}
It is not hard to see that this limit exists. From above, Caro, Patk\'{o}s, and Tuza~\cite{caro2022connected} found a family of trees whose connected extremal number is asymptotically smaller, yielding $\gamma\le \frac{2}{3}$. From below, for every tree, they gave constructions showing that $\gamma\ge \frac{1}{3}$.
They asked where the truth lies between $\frac{1}{3}$ and $\frac{2}{3}$.    

Our main result settles this problem.

\begin{theorem}\label{thm:gamma}
Let $\gamma$ be as defined in~\eqref{Eq:Def_gamma}, we have $\gamma=\frac{1}{2}$.
\end{theorem}

To obtain the lower bound $\gamma\ge \frac{1}{2}$, we provide several families of different constructions depending on the `center of mass' of the forbidden tree (see Section~\ref{sec:barycenter}), realizing the following.  

\begin{theorem}\label{thm:Main}
    For any $k$-vertex tree $T_k$, we have 
    \[\ex_c(n,T_k)\geq \left( \frac{1}{4}-o_k(1) \right)kn.\]
\end{theorem}

On the other hand, we determine the exact connected extremal number  of brooms with $k$ vertices and diameter $d$, denoted by $B(k,d)$, for large enough $n$. In particular, $B(k,d)$ is the graph obtained from a path of $d+1$ vertices by blowing up a leaf to an independent set of size $k-d$. The following theorem is stated using graphs $\textsf{G}_{n,\cdot,\cdot}$ and $\textsf{F}_{n,\cdot,\cdot}$, which are defined in Section~\ref{Sec:consttructions}. Some of these graphs (so-called edge blow-up of stars) have been studied before, see e.g.~\cite{chen2003extremal,erdos1995extremal,Yuan2022extremal}. As the path is also a broom, the result below can be viewed as an extension of Kopylov's result \eqref{Eq:Erd_gallai_connected}.

\begin{theorem}\label{thm:brooms}
For every integer $k$ and $d$ such that  $k\geq d+2 \geq 8$, and $n\geq k^{dk}$ we have
  \[
\begin{displaystyle}
\ex_c(n, B(k,d)) =
\begin{cases}
e(\emph{\textsf{G}}_{n,d,\floor{\frac{d-1}{2}}}) \,\,\:\quad\quad\quad\quad\quad\quad\quad\quad\quad\quad\quad\mbox{if  $d\geq \frac{k+5}{2}$}, \\
\max\{e(\emph{\textsf{G}}_{n,d,\floor{\frac{d-1}{2}}}), e(\emph{\textsf{F}}_{n,\frac{k+2}{2},1})\} \;\;\;\quad\quad\quad\mbox{if d=$\frac{k+2}{2} \mbox{\ or } \frac{k+4}{2}$}, \\
\max\{e(\emph{\textsf{G}}_{n,d,\floor{\frac{d-1}{2}}}), e(\emph{\textsf{F}}_{n,d,2}), e(\emph{\textsf{F}}_{n,d,3})\} \quad\mbox{if $d= \frac{k+3}{2}$}, \\
\floor{\frac{(k-d)n}{2}} \,\:\:\:\:\:\quad\quad\quad\quad\quad\quad\quad\quad\quad\quad\quad\quad \mbox{if $d\leq  \frac{k+1}{2}$}.
 \end{cases}
 \end{displaystyle}
 \]
\end{theorem}

As a corollary, we get the matching upper bound $\gamma\le\frac{1}{2}$ as $\ex_{c}(n,B(k,\ceil{\frac{k}{2}}))= \left( \frac{1}{4}+o_k(1) \right)kn$.

\medskip

\noindent\textbf{Organization.} The rest of the paper is organized as follows. Preliminaries and various constructions are given in Section \ref{sec:prelim}. We then present the proofs of the two main results, Theorems \ref{thm:Main} and \ref{thm:brooms} in Sections \ref{sec:proof} and \ref{sec:broom} respectively. Some concluding remarks are given in Section \ref{sec:concluding remarks}.

\section{Preliminaries}\label{sec:prelim}
\subsection{The Barycenter of a tree}\label{sec:barycenter}
We introduce a key notion needed for constructions in the proof of Theorem~\ref{thm:Main}. For any tree $T$ on $k$ vertices, we call a vertex $v$ of $T$ \textit{a barycenter} if $v$ belongs to a largest connected component of $T-e$ for every edge $e$ in $T$, that is, the vertex $v$ belongs to the component of size at least $\lceil\frac{k}{2}\rceil$ in the graph obtained from $T$ by removing an edge $e$. 
\begin{proposition}\label{claim:barycenter}
    Every tree has either a unique barycenter, or there are exactly two barycenters in the tree joined by an edge.
\end{proposition}
\begin{proof}
    If $T$ contains an edge $e$ such that the two components of $T-e$ have the same size, then it is easy to see that  the vertices of $e$ are the only barycenters of $T$. 
    
    If such an edge does not exist in $T$, then let $e^\prime$ be an edge such that the size of the larger connected component of $T-e^\prime$ is minimum. By the choice of $e^\prime$, the vertex $w$ of $e^\prime$ belonging to the larger connected component of $T - e^\prime$ is a barycenter. Considering edges adjacent  to $w$ we can exclude every other vertex from being a barycenter of $T$ by the choice of $w$.
\end{proof}

\subsection{Constructions of various classes of graphs}\label{Sec:consttructions}
We begin with constructions for Theorem~\ref{thm:Main}.

\medskip

\noindent\textbf{The family $\textsf{G}_{n,\cdot,\cdot}$.} We first define the family of extremal graphs for~\eqref{Eq:Erd_gallai_connected}. Recall `$ \cup $'  denotes the disjoint union of graphs, `$+$' denotes the join of the graphs, and  $\overline{K}_{t}$ denotes an independent set of size $t$. For $n\geq k \geq 2s$ let $\textsf{G}_{n,k,s}:=\left(K_{k-2s}\cup \overline{K}_{n-k+s} \right)+K_s$, see Figure~\ref{fig:Constructions}. Note that for every $n$ and $k\geq 6$, there exists a constant $a$ such that $a<2k$ and the only extremal graphs achieving equality in~\eqref{Eq:Erd_gallai_connected} are $\textsf{G}_{n,k-1,1}$ for $n\leq a$ and $\textsf{G}_{n,k-1,\floor{\frac{k-2}{2}}}$ for $n\geq a$.

\noindent\textbf{The families $\textsf{S}_{n,x}$ and $\textsf{P}_{n,x}$.} Let $x$ be an integer such that  $\frac{k}{2}<x<k$ or $x=\floor{\frac{k-2}{2}}$. For the sake of the simplicity of the write-up, we denote 
\begin{equation*}
\begin{displaystyle}
a_x:= 
\begin{cases}
 \floor{\frac{2x^2}{k}}-2 \:\quad\quad\mbox{if  $\frac{k}{2}<x<k$, and } \\
x \;\;\;\qquad\quad\quad\quad\mbox{if $x=\floor{\frac{k-2}{2}}$.}
\end{cases}
 \end{displaystyle}
\end{equation*}

\begin{figure}[h]
\centering    
\begin{tikzpicture}[scale=0.53]
\draw [line width=1pt] (9,14) circle (0.9808159868191375cm);
\draw [line width=1pt] (12,14) circle (0.6529931086925803cm);
\draw [line width=1pt] (16,14) circle (1.882012881774223cm);

\draw [line width=1pt] (11.973913114729054,14.652471818867957)-- (9.105373645040958,14.97513916695556);
\draw [line width=1pt] (11.999514451672207,13.347007071827862)-- (9.113297868919155,13.02574972779148);

\draw [line width=1pt] (15.53384696202449,12.165)-- (11.999514451672207,13.347007071827862);
\draw [line width=1pt] (11.973913114729054,14.652471818867957)-- (15.545126412345173,15.84);

\draw (7.739,14.610374761682377) node[anchor=north west] {$K_{k-2s}$};
\draw (11.307117645824812,14.610374761682377) node[anchor=north west] {$K_s$};
\draw (14.202193621681706,14.610374761682377) node[anchor=north west] {$\overline{K}_{n-k+s}$};
\draw (10.7,14) node[anchor=center] {$\cdots$};
\draw (13.4,14) node[anchor=center] {$\cdots$};
\draw (13.3,10.3) node[anchor=center] {$$};
\end{tikzpicture}
\begin{tikzpicture}[scale=0.69]
\fill (5.5,20.975) circle (0.07cm);
\fill (11.5,20.975) circle (0.07cm);
\fill (14,20.975) circle (0.07cm);
\fill (17.5,20.975) circle (0.07cm);
\draw [line width=1pt] (5,17) circle (1cm);
\draw [line width=1pt] (8.5,17) circle (1cm);
\draw [line width=1pt] (14,17) circle (1cm);
\draw [line width=1pt] (17.5,17) circle (1cm);
\draw[decorate, decoration=snake, segment length=6.5mm, line width=1pt] 
    (5.5,21)-- (2.5,17);
\draw [decorate, decoration=snake, segment length=6.5mm, line width=1pt]  (5.5,21)-- (5,18);
\draw [decorate, decoration=snake, segment length=6.5mm, line width=1pt]  (5.5,21)-- (8.5,18);
\draw [decorate, decoration=snake, segment length=6.5mm, line width=1pt] (11.5,21)-- (11.5,17);
\draw [decorate, decoration=snake, segment length=6.5mm, line width=1pt] (14,21)-- (14,18);
\draw [decorate, decoration=snake, segment length=6.5mm, line width=1pt] (17.5,21)-- (17.5,18);
\draw [line width=1pt] (11.5,21)-- (15,21);
\draw [line width=1pt] (16.5,21)-- (17.5,21);
\draw (15.7,21) node[anchor=center] {$\cdots$};
\draw (5,17) node[anchor=center] {$K_{x}$};
\draw (8.5,17) node[anchor=center] {$K_{x}$};
\draw (14,17) node[anchor=center] {$K_{x}$};
\draw (17.5,17) node[anchor=center] {$K_{x}$};
\draw (3.5,19) node[anchor=center ] {$P$};
\draw (11,19) node[anchor=center ] {$P$};
\draw (6.3,19) node[anchor=center ] {$P_{a_x-x+2}$};
\draw (8.8,19) node[anchor=center] {$P_{a_x-x+2}$};
\draw (13,19) node[anchor=center ] {$P_{a_x-x+2}$};
\draw (16.5,19) node[anchor=center] {$P_{a_x-x+2}$};
\draw (5.5,21.2) node[anchor=center] {$w$};
\draw (11.5,21.2) node[anchor=center] {$w_0$};
\draw (14,21.2) node[anchor=center] {$w_1$};
\draw (6.7,17) node[anchor=center] {$\cdots$};
\draw (15.7,17) node[anchor=center] {$\cdots$};
\end{tikzpicture}
\caption{The graph  $\textsf{G}_{n,k,s}$ on the left, the graph $\textsf{S}_{n,x}$ in the middle  and the graph $\textsf{P}_{n,x}$ on the right.}
\label{fig:Constructions}
\end{figure}


Let $\textsf{S}_{n,x}$ be a graph consisting of  $\floor{\frac{n-1}{a_x}}$ vertex disjoint $K_x$ with pendant paths of length $a_x-x$, a path of $n-1-a_x\floor{\frac{n-1}{a_x}}$ vertices and a vertex $w$ adjacent with an endpoint of each of these paths. 

Let $\textsf{P}_{n,x}$ be a graph consisting of  $\floor{\frac{n}{a_x+1}}$ vertex disjoint $K_x$ with pendant paths of length $a_x-x+1$ with the terminal leaf $w_i$ ($i\geq 1$), a path of $n-(a_x+1)\floor{\frac{n}{a_x+1}}$ vertices with a terminal leaf $w_0$, such that $w_0w_1\dots w_{\floor{\frac{n}{a_x+1}}}$ is a path.

It is easy to see that $e(\textsf{S}_{n,x})=(\frac{1}{4}-o_k(1))kn$ and $e(\textsf{P}_{n,x})=(\frac{1}{4}-o_k(1))kn$ for all $x$. Now we provide constructions for the lower bound in Theorem~\ref{thm:brooms}.

\medskip

\noindent\textbf{The family $\textsf{F}_{n,\cdot,\cdot}$.} 
For $n>d\geq 2$, let $\textsf{F}_{n,d,1}$ be an $n$-vertex connected graph such that  every maximal $2$-connected block is $K_{d-1}$ except at most one clique of size $n-\floor{\frac{n-1}{d-2}}(d-2)$, all sharing a common vertex.
Thus, if $n-1=(d-2)p_1+q_1$ such that $0\leq q_1< d-2$, then 
\[
e(\textsf{F}_{n,d,1})=p_1\binom{d-1}{2}+\binom{q_1+1}{2}=\frac{(d-1)(n-1)}{2}-\frac{q_1(d-2-q_1)}{2}.
\]

Let $\textsf{F}_{n,d,2}$ be an $n$-vertex connected graph such that  every maximal $2$-connected block is a clique with one of size $d-1$, the rest of size $d-2$ except at most one clique of size $n-1-\floor{\frac{n-2}{d-3}}(d-3)$, all sharing a common vertex.
Thus, if $n-2=(d-3)p_2+q_2$ such that $p_2\geq 1$ and $0\leq q_2<d-3$, then 
\[
e(\textsf{F}_{n,d,2})=p_2\binom{d-2}{2}+d-2+\binom{q_2+1}{2}=\frac{(d-2)n}{2}-\frac{q_2(d-3-q_2)}{2}.
\]


For an even integer $d$, let $n-1=(d-3)p_3+q_3$ such that $0\leq q_3<d-3$. 
If $p_3\geq q_3$, let $\textsf{F}_{n,d,3}$ be an $n$-vertex connected graph such that it contains $p_3$ maximal $2$-connected blocks $G_1,\ldots, G_{p_3}$ all sharing a common vertex $v$ with $p_3-q_3$ of them being $K_{d-2}$. The rest of the maximal $2$-connected blocks $G_i$ are the cliques of size $d-1$ without a perfect matching of $G_i-v$. 
We have
\[
e(\textsf{F}_{n,d,3})=\frac{(d-2)(n-1)}{2}.
\]
If $p_3<q_3$, let $\textsf{F}_{n,d,3}$ be an $n$-vertex connected graph such that it contains $p_3+1$ maximal $2$-connected blocks $G_1, \ldots, G_{p_3+1}$ all sharing a common vertex $v$ with $p_3$ of them being the cliques of size $d-1$ without a perfect matching of $G_i-v$ and the remaining one is a clique of size $q_3-p_3+1$.
We have
\[
e(\textsf{F}_{n,d,3})=p_3\frac{(d-2)^2}{2}+\binom{q_3-p_3+1}{2}=\frac{(d-2)(n-1)}{2}-\frac{(q_3-p_3)(d-3-(q_3-p_3))}{2}.
\]

\subsection{Other tools}
To prove the upper bound of $\ex_c(n,B(k,d))$ we need the following theorems.
Let $\C_{\geq k}$ denote the class of cycles of length at least $k$. 
For a class of graphs $\F$, the extremal number of $\F$ is defined analogously, i.e.~the maximum number of edges in a graph not containing a subgraph $F$ for all $F\in \F$, denoted by $\ex(n,\F)$.

Woodall~\cite{woodall1976maximal} and  Kopylov~\cite{kopylov1977maximal} independently improved the Erd\H{o}s-Gallai theorem~\cite{erdHos1959maximal} for long cycles and obtained the following exact result for every $n$, see also~\cite{faudree1975path}.

\begin{theorem}[Woodall~\cite{woodall1976maximal}, Kopylov~\cite{kopylov1977maximal}] \label{Thm:Woodall-cycle}
Let $n=p(k-2)+q+1$, where $0\leq q<k-2$ and $k\geq 3$, $p\geq 1$,
 \[
\ex(n,\C_{\geq k}) = p\binom{k-1}{2}+\binom{q+1}{2}=\frac{(k-1)(n-1)}{2}-\frac{q(k-2-q)}{2}.
\]

\end{theorem}
Note that $\textsf{F}_{n,k,1}$ is an extremal graph.

\begin{theorem}[Kopylov~\cite{kopylov1977maximal}, Woodall~\cite{woodall1976maximal}, Fan, Lv and Wang~\cite{fan2004cycles}]\label{thm:2-connected_Kopylov}
    Suppose $n\geq k\geq 5$, then every $2$-connected $n$-vertex $\mathcal{C}_{\geq k}$-free graph contains at most 
 \begin{equation*}
\max \left\{\binom{k-2}{2}+2(n-(k-2)), \floor{\frac{k-1}{2}}\left(n-\ceil{\frac{k+1}{2}}\right)+\binom{\ceil{\frac{k+1}{2}}}{2} \right\} \mbox{~edges.}
\end{equation*}
The extremal graphs are $\emph{\textsf{G}}_{n,k,2}$ and $\emph{\textsf{G}}_{n,k,\floor{\frac{k-1}{2}}}$.
\end{theorem}


\section{Proof of Theorem~\ref{thm:Main}}\label{sec:proof}
Let $v$ be a barycenter of the tree $T_k$, which exists by Proposition~\ref{claim:barycenter}. 
We split the proof into three parts depending on the degree of  $v$. Recall that $a_x=\floor{\frac{2x^2}{k}}-2$ if $\frac{k}{2}<x<k$ and $a_x=x$ if $x=\floor{\frac{k-2}{2}}$.

\smallskip

\noindent{\bf Case 1.}  $d(v)=2$. 
Let $x=\floor{\frac{k-2}{2}}$. We claim that $\textsf{S}_{n, x}$ is $T_k$-free and the desired lower bound follows. 
Consider a possible embedding  of $T_k$ in $\textsf{S}_{n, x}$. 
Since $d(v)=2$ and the size of each clique is $x$ in $\textsf{S}_{n, x}$, $v$ is not mapped to $w$ as $2x+1\leq k-1$.
On the other hand, since every clique or path $P$ has size strictly less than 
$\frac{k}{2}$ and there is an edge separating the clique or path $P$ from the rest of the graph, 
the vertex $v$ is not mapped to a vertex of  the clique or $P$ by the definition of the barycenter. 
Thus $\textsf{S}_{n, x}$ is $T_k$-free.

\smallskip

\noindent{\bf Case 2.}  $d(v)\geq 4$.  
Let $x=\floor{\frac{k-2}{2}}$. We claim that $\textsf{P}_{n, x}$ is $T_k$-free and the desired lower bound follows. 
Consider a possible embedding  of $T_k$ in $\textsf{P}_{n, x}$. 
Since $d(v)=4$ and the degree of each vertex 
not in the cliques is at most three, $v$ is not mapped to any $w_i$ or any vertex of $P$.
On the other hand, since every clique has a size strictly less than 
$\frac{k}{2}$ and there is an edge separating the clique from the rest of the graph, 
$v$ is not mapped to a vertex of  a clique  by the definition of the barycenter.
Thus  $\textsf{P}_{n, x}$ is $T_k$-free.

\smallskip

\noindent{\bf Case 3.} $d(v)=3$.
First, we observe the following.

\begin{claim}\label{claim:bounded degree}
    If $T_k$ contains a sub-tree $S$ with $\Delta(S)\leq a$.
    Then we have
    \[\ex_c(n,T_k)\geq \left( \frac{v(S)}{ak}-o_k(1) \right)kn.\]
\end{claim}
\begin{poc}
    Since $\Delta(S)\leq a$ and $e(S)=v(S)-1$, each bipartite class of $S$ has size at least $\frac{v(S)-1}{a}$. Hence, each bipartite class of $T_k$ also has size at least $\frac{v(S)-1}{a}$. Thus $K_{\floor{\frac{v(S)-2}{a}},n-\floor{\frac{v(S)-2}{a}}}$ is $T_k$-free, and we have $\ex_c(n,T_k)\geq e(K_{\floor{\frac{v(S)-2}{a}},n-\floor{\frac{v(S)-2}{a}}})=\left( \frac{v(S)}{ak}-o_k(1) \right)kn.$
\end{poc}

\begin{claim}\label{claim:Balanced}
    Let $c$ be a constant and let $S_1$ and $S_2$ be vertex disjoint sub-trees of $T_k$ such that $v(S_1)+v(S_2)\geq \frac{k-c}{2}$. If each of $S_1$ and $S_2$ is a spider\footnote{A spider is a tree with all vertices of degree at most two, except possibly one vertex of any degree, referred to as the central vertex of the spider.} with  the central vertex of degree at most three.
    Then we have    \[\ex_c(n,T_k)\geq \left( \frac{1}{4}-o_k(1) \right)kn.\]
\end{claim}
\begin{poc}
    By the hypothesis on $S_1, S_2$, it is easy to see that each bipartite class of $T_k$ has size at least $\frac{k-c}{4}-4$. Thus, $K_{\floor{\frac{k-c}{4}}-5,n-\left(\floor{\frac{k-c}{4}}-5\right)}$  is $T_k$-free, and thus $\ex_c(n,T_k)\geq e(K_{\floor{\frac{k-c}{4}}-5,n-\left(\floor{\frac{k-c}{4}}-5\right)})=\left( \frac{1}{4}-o_k(1) \right)kn.$
\end{poc}
Let $T_k-v$ be a forest of trees $T_1,T_2$, and $T_3$.
By the definition of barycenter, the number of vertices in the union of any two of the trees $T_1,T_2,T_3$ is at least $\frac{k-2}{2}$. 
Thus, at most one of these trees is a path since otherwise by Claim~\ref{claim:Balanced} the desired lower bound follows. 
We assume also that  $\textsf{S}_{n,x}$ and $\textsf{P}_{n,x}$ contain $T_k$ as a subgraph for all $x$ with $\frac{k}{2}<x<k$ for otherwise the desired lower bound follows. 

\begin{claim}\label{claim:map to clique}
    Let $x$ be a number such that $a_x-x+2\geq \frac{k-c}{4}$, for some constant $c$. If there is an embedding of $T_k$ in $\textsf{S}_{n,x}$ or $\textsf{P}_{n,x}$, then $v$ must be mapped to a vertex of a clique in $\textsf{S}_{n,x}$ or $\textsf{P}_{n,x}$. 
\end{claim}
\begin{poc}
Since $d(v)=3$, we only need to prove that $v$ is  not  mapped to $w$ in $\textsf{S}_{n,x}$ or $v$ is  not  mapped to $w_i$ in $\textsf{P}_{n,x}$. 
 Note first that $v$ is  not  mapped to $w$ in $\textsf{S}_{n,x}$, otherwise, as at least two of the trees $T_1,T_2,T_3$ are not paths, implying that the images of two of them, say $T_1$ and $T_2$, must contain a vertex of some clique in $\textsf{S}_{n,x}$. Hence, $T_1$ and $T_2$ both contain a sub-path of length at least $\frac{k-c}{4}-2$ from the vertex $w$ to a vertex of a clique, 
and we are done by Claim~\ref{claim:Balanced}.
Similarly, for every embedding of $T_k$ in $\textsf{P}_{n,x}$ that maps $v$ to some $w_i$, at most one of the trees $T_1, T_2, T_3$ contains a vertex in the cliques, otherwise two of them contain a sub-path of length at least $\frac{k-c}{4}-2$ and we are done by Claim \ref{claim:Balanced}. Thus we may assume only $T_1$ contains the vertices in the cliques. Let $A$ be the set of vertices of $T_1$ embedded in the cliques. Thus $T_k-A$ is a sub-tree of $T_k$ of order at least $\frac{k-2}{2}+(a_x-x+1)\geq \frac{3k-c-8}{4}$ and maximum degree at most 3. By Claim \ref{claim:bounded degree} with $T_k-A$ playing the role of $S$, we have $\ex_c(n,T_k)\geq \left( \frac{1}{4}-o_k(1) \right)kn$. Hence, $v$ is not mapped to $w_i$ in $\textsf{P}_{n,x}$.
\end{poc}

\begin{figure}[ht]
\centering    
\begin{tikzpicture}[scale=0.9]
\fill (5.5,20.975) circle (0.07cm);
\draw [line width=1pt] (2,17) circle (1cm);
\draw [line width=1pt] (5,17) circle (1cm);
\draw [line width=1pt] (8.5,17) circle (1cm);
\draw[decorate, decoration=snake, segment length=6.5mm, line width=1pt] 
    (5.5,21)-- (0.5,19.5);
\draw [decorate, decoration=snake, segment length=6.5mm, line width=1pt, color=red]  (5.5,21)-- (2,17.8);
\draw [decorate, decoration=snake, segment length=6.5mm, line width=1pt, color=red]  (5.5,21) -- (5.2,19);
\draw [decorate, decoration=snake, segment length=6.5mm, line width=1pt, color=red, dashed]  (5.5,21)-- (7.6,19);
\draw [decorate, decoration=snake, segment length=6.5mm, line width=1pt]  (5.2,19) -- (5,18);
\draw [decorate, decoration=snake, segment length=6.5mm, line width=1pt]  (7.6,19)-- (8.5,18);
\draw (2,17.5) node[anchor=center] {{\color{blue}$v$}};
\draw (2,17) node[anchor=center] {$K_{x_1}$};
\draw (5,17) node[anchor=center] {$K_{x_1}$};
\draw (8.5,17) node[anchor=center] {$K_{x_1}$};
\draw (1,20) node[anchor=center ] {$P$};
\draw (1.7,19) node[anchor=center ] {$P_{a_{x_1}-x_1+2}$};
\draw (6.2,19) node[anchor=center ] {$P_{a_{x_1}-x_1+2}$};
\draw (8.7,19) node[anchor=center] {$P_{a_{x_1}-x_1+2}$};
\draw (5.8,21.2) node[anchor=center] {$w$ ({\color{red}$u$})};
\draw (6.7,17) node[anchor=center] {$\cdots$};

\draw [dotted][line width=1pt] (11.3,19) circle (1cm);
\fill (11.3,19) circle (0.07cm);
\draw (11.3,18.7) node[anchor=center] {$v$};
\fill (13,19) circle (0.07cm);
\fill (14,19) circle (0.07cm);
\draw (13,18.7) node[anchor=center] {$v_1$};
\draw (14,18.7) node[anchor=center] {$v_2$};
\draw (15.3,19.5) node[anchor=center] {$P_{\geq a_{x_1}-x_1+1}$};
\draw [line width=1pt] (13,19)-- (14,19);
\draw [line width=1pt] (11.9,19.5)-- (13,19);
\draw [line width=1pt] (11.9,18.5)-- (13,19);
\fill (16.5,19) circle (0.07cm);
\draw (16.5,18.7) node[anchor=center] {$u$};
\draw [decorate, decoration=snake, segment length=6.5mm, line width=1pt, color=red]  (14,19)-- (16.5,19);
\draw [dotted][decorate, decoration=snake, segment length=6.5mm, line width=1pt, color=red]  (16.5,19)-- (18,20);
\draw [decorate, decoration=snake, segment length=6.5mm, line width=1pt, color=red]  (16.5,19)-- (18,18);
\draw[decorate,decoration={brace,mirror,raise=1pt,amplitude=0.4cm}] (14,18) -- (18.1,18)node[black,midway,yshift=-0.7cm]{$S_y$};
\end{tikzpicture}
\caption{The graph $T_k$ in $\textsf{S}_{n,x_1}$ on the left and the graph $T_k$ on the right.}
\label{fig:ConstructionsII}
\end{figure}

Let $x_1=\floor{\frac{k}{\sqrt{2}}}$ then we have $a_{x_1}-x_1+2\geq k-x_1-4\geq \frac{k}{4}-4$.
By Claim \ref{claim:map to clique}, every embedding  of $T_k$ in $\textsf{S}_{n,x_1}$ maps  $v$  to a vertex of a clique.
Furthermore, since the components of $\textsf{S}_{n,x_1}-w$ are strictly smaller than $k$, there must be a vertex $u$ of $T_k$ distinct from $v$ mapped to $w$. 
Thus $T_k$ contains a spider with central vertex $u$ as a subgraph with a leg of length at least $a_{x_1}-x_1+1$, see Figure \ref{fig:ConstructionsII}. 
Indeed, $T_k$ does not contain a pair of vertices from different cliques in $\textsf{S}_{n,x_1}$, otherwise there are two vertex disjoint sub-paths of $T_k$ such that each of them contains $a_{x_1}-x_1\geq \frac{k}{4}-6$ vertices and we are done by Claim~\ref{claim:Balanced} as a path is a spider with the central vertex of degree two.
On the other hand, as $T_k$ is a subgraph of  $\textsf{P}_{n,x_1}$, with the same argument $T_k$ does not contain a pair of vertices from different cliques in $\textsf{P}_{n,x_1}$. 
Thus $T_k$ contains an edge, say $v_1v_2$, such that after removing it from $T_k$ we get a component where the degrees of the vertices are at most three and which contains a sub-path of length  $a_{x_1}-x_1-1$, each vertex of this path has degree at most two in~$T_k$. 
Hence, by the previous argument in this paragraph, $T_k-v_1v_2$ contains a component denoted by $S_y$ (containing the vertex $v_2$) isomorphic to a spider with $y$ vertices and  at most three legs, with one leg of length at least $a_{x_1}-x_1$, see Figure~\ref{fig:ConstructionsII}.  
 Without loss of generality, by choosing  $S_y$ such that $y$ is the maximum possible, we may assume that the degree of $v_1$ in $T_k$ is at least three.

Now, we consider $x_2=k-y-1>\frac{k}{2}$ and an embedding  of $T_k$ in $\textsf{S}_{n,x_2}$.
If $v$ is mapped to $w$, then there are two vertex disjoint paths of length $a_{x_2}-x_2$ from $v$ since at least two of the trees $T_1,T_2,T_3$ are not isomorphic to a path. 
Furthermore, one of them is vertex disjoint from $S_y$ and  we are done by Claim~\ref{claim:Balanced}, since $a_{x_2}-x_2+1+y\geq \frac{2(k-y-1-k/2)^2}{k}+\frac{k}{2}-3\geq\frac{k}{2}-3$ for $x_2=k-y-1$.
If $v$ is mapped to a vertex of a clique then either the entire $S_y$ including the vertex $v_1$ is mapped outside of the  clique containing the vertex $v$, or inside of the clique.  
Indeed, If  $v_1$ is mapped in the same clique as $v$ and part of $S_y$ is mapped out of that clique, then the number of vertices outside of the clique is at most $y$ while the size of the clique is $k-y-1$, thus $v(T_k)\leq k-1$, a contradiction.
If $S_y$ with the vertex $v_1$ is mapped entirely outside of the clique then the vertex $v_1$ is mapped to the vertex $w$ or to a vertex of a different clique thus there is a sub-path vertex disjoint from $S_y$, from the clique to $w$,  of length  $a_{x_2}-x_2+1$.  
If $S_y$ is mapped entirely inside of the clique then  $T_k$ contains a sub-path vertex disjoint from $S_y$ of length  $a_{x_2}-x_2+1$, since there is a vertex of $T_k$ mapped to the vertex  $w$. 
In both cases we are done by Claim~\ref{claim:Balanced}, since $a_{x_2}-x_2+2+y>\frac{k}{2}-c$.

\section{Proof of Theorem \ref{thm:brooms}}\label{sec:broom}
For the lower bound, let $n>k-2\geq d\geq 2$. First, we consider $\textsf{G}_{n,d,\floor{\frac{d-1}{2}}}$ and an almost $(k-d)$-regular graph $\textsf G^\prime$, which is $(k-d)$-regular graph if $(k-d)n$ is even, and if $(k-d)n$ is odd then every vertex of $\textsf G^\prime$ has degree $k-d$ except one vertex of degree $k-d-1$. Clearly, $\textsf{G}_{n,d,\floor{\frac{d-1}{2}}}$ is $P_{d+1}$-free, thus it does not contain $B(k,d)$ as a subgraph. The almost $(k-d)$-regular graph $\textsf G^\prime$ is $B(k,d)$-free as it does not contain a vertex of degree at least $k-d+1$ and for every $n>k-d$ the almost ($k-d$)-regular graph $\textsf{G}^\prime$ always exists. 
Indeed, by considering appropriate parameters for circulant graphs, we can construct regular graphs. For example, let $G^\prime$ be a graph with $V(G^\prime)=\{v_0,v_1,\ldots, v_{n-1}\}$ and $E(G^\prime)=\{v_iv_{i+j}: v_i\in V(G^\prime) \mbox{\ and\ }  1\leq j\leq \frac{k-d}{2}\}$ if $k-d$ is even and $E(G^\prime)=\{v_iv_{i+j}: v_i\in V(G^\prime) \mbox{\ and\ }  1\leq j\leq \frac{k-d+1}{2}\}\setminus \{v_iv_{i+1}: v_i\in V(G^\prime) \mbox{\ and\ } i \mbox{\ is even}\}$ if $k-d$ is odd, where $v_{i+j}=v_{i+j\mod n}$. Clearly, $G^\prime$ is an almost $(k-d)$-regular graph and $G^\prime$ is connected.
The number of edges in these graphs are $e(\textsf G^\prime)=\floor{\frac{(k-d)n}{2}}$ and 
\[
e(\textsf{G}_{n,d,\floor{\frac{d-1}{2}}})=\floor{\frac{d-1}{2}}\left(n-\ceil{\frac{d+1}{2}}\right)+\binom{\ceil{\frac{d+1}{2}}}{2}.
\]

Now we consider graphs $\textsf{F}_{n,\frac{k+2}{2},1}$ if $d=\frac{k+2}{2}$ or $\frac{k+4}{2}$, $\textsf{F}_{n,d,2}$ and $\textsf{F}_{n,d,3}$ if $d=\frac{k+3}{2}$. We claim that these graphs are $B(k,d)$-free. Indeed, if one of these graphs contains a copy of $B(k,d)$, say $B$, and $u$ is the vertex with degree $k-d+1$ in $B$, then $u$ can not be embedded at $v$ as the size of every maximal $2$-connected block is at most $d-1$, where $v$ is the common vertex of all maximal $2$-connected blocks. So $B$ is contained in two maximal $2$-connected blocks, which is impossible as the size of any two maximal $2$-connected blocks of $\textsf{F}_{n,\frac{k+2}{2},1}$ and $\textsf{F}_{n,\frac{k+3}{2},2}$ is at most $\max\{2(\frac{k+2}{2}-1)-1, (\frac{k+3}{2}-1)+(\frac{k+3}{2}-2)-1\}=k-1$ and every vertex of  $\textsf{F}_{n,\frac{k+3}{2},3}-v$ has degree at most $d-3=k-d$.

So for every $n>k-2\geq d\geq 2$, we get the lower bound of Theorem \ref{thm:brooms}.

For the matching upper bound, let us consider  a connected graph $G$ with $n$ vertices, such that $n\geq k^{dk}$,  not containing  $B(k,d)$ as a subgraph and let $v$ be a vertex of $G$ such that $d(v)=\Delta(G)$.
    
    If $\Delta(G)\leq k-d$ then we have $e(G)\leq \floor{\frac{(k-d)n}{2}}$.

    If $k-d+1 \leq \Delta(G)\leq k-2$, then we will find a copy of $B(k,d)$ in $G$ with a greedy argument for every $n\geq k^{dk}$ resulting in a contradiction. 
First, we show that there exists a path of at least $dk$ vertices starting at $v$. 
Since if one considers a breadth-first search tree from the vertex $v$, the size of $i$-th neighborhood contains at most $k^i$ vertices.
Thus by the connectivity of $G$, we have a path of $dk$ vertices starting at  the vertex $v$.
Let $P=vu_2u_3\cdots u_{dk}$ be a path of order $dk$ starting at $v$, a vertex of maximum degree. Note that, we have $u_2\in N(v)$. 
By pigeonhole principle and since  $d(v)\leq k-2$, there exists a sub-path $P^\prime=u_iu_{i+1}\cdots u_{i+d-2}$ of $P$ such that $V(P^\prime)\cap N(v)=\{u_i\}$. Then $G[V(P^\prime)\cup N(v)\cup \{v\}]$ contains a copy of $B(k,d)$ as $d(v)\geq k-d+1$, a contradiction.

    If $\Delta(G)\geq k-1$, then  $G$ is $\C_{\geq d}$-free since otherwise there is a shortest path $P$ from $v$ to the longest cycle $C$ of $G$, such that $\abs{V(P)\cap V(C)}=1$ and $G[V(C)\cup V(P)\cup N(v)]$ contains a copy of $B(k,d)$.
If $G$ is $2$-connected by Theorem~\ref{thm:2-connected_Kopylov} we have the number of edges in $G$ is at most
$e(\textsf{G}_{n,d,\floor{\frac{d-1}{2}}})$ since $n$ is large. From here we assume $G$ is not $2$-connected and its maximal $2$-connected blocks are $G_1,G_2,\dots,G_s$. 
Let $n_i$ be the number of vertices of $G_i$ for every $i \in [s]$ and we assume $n_1\geq n_2\geq n_3\cdots \geq n_s$. 
Note that we have $\sum_{i=1}^{s}n_i=n+(s-1)$.
  By Theorem~\ref{Thm:Woodall-cycle}  the circumference of $G$ is at least $d-2$ or we have the desired upper bound.
Thus the circumference of $G$ is either $d-2$ or $d-1$.

\smallskip

\noindent{\bf Case 1.} The circumference of $G$ is $d-2$.
By Theorem~\ref{thm:2-connected_Kopylov}, for each $i\in [s]$ and $n_i\geq d-1\geq 5$ we have $e(G_i)\leq \max\{e(\textsf{G}_{n_i,d-1,2}), e(\textsf{G}_{n_i,d-1,\floor{\frac{d-2}{2}}})\}$.  It is easy to see that there exists a constant $a<2d$ such that  $e(\textsf{G}_{n_i,d-1,2})\leq e(\textsf{G}_{n_i,d-1,\floor{\frac{d-2}{2}}})$ if and only if $n_i\geq a$. Here $a=\frac{5d-16}{4}$ if $d$ is even and $a=\frac{5d-13}{4}$ if $d$ is odd. Thus we have
\begin{equation}\label{eqcase1}
\begin{displaystyle}
e(G_i) \leq 
\left.
\begin{cases}
 \floor{\frac{d-2}{2}}\left(n_i-\ceil{\frac{d}{2}}\right)+\binom{\ceil{\frac{d}{2}}}{2} \:\:\:\,\quad\mbox{if  $n_i\geq \max\{a, d-1\}$,} \\
\binom{d-3}{2}+2(n_i-(d-3))\:\:\,\,\quad\quad\mbox{if $d-1 \leq n_i< a$,} \\
\binom{n_i}{2} \,\quad\quad\quad\quad\quad\quad\quad\quad\quad\quad\quad \mbox{if $n_i\leq  d-2$.}   
\end{cases}
 \right\}
 \leq \floor{\frac{d-1}{2}}(n_i-1), 
 \end{displaystyle}
\end{equation}
and the latter equality holds if and only if $n_i=d-2$ and $d$ is even.

Let $\ell=\floor{\frac{d-1}{2}}\ceil{\frac{d+1}{2}}-\binom{\ceil{\frac{d+1}{2}}}{2}$. Note that since $n$ is large, $\ell \leq s$ if $n_1<\max\{a, d-1\}$.   
We claim that if one of the following statements holds, then $e(G)\leq e(\textsf{G}_{n,d,\floor{\frac{d-1}{2}}})$.

(i) $n_1\geq \max\{a, d-1\}$,

(ii) $n_1<a$ and $n_{\ell}\geq d-1$,

(iii) the size of $I=\{i\in [s]: n_i\leq d-2 \mbox{ and } e(G_i)\leq \binom{d-2}{2}-1\}$ is at least $\ell$.

If (i) holds, then by (\ref{eqcase1}) we have $e(G)\leq \floor{\frac{d-2}{2}}\left(n_1-\ceil{\frac{d}{2}}\right)+\binom{\ceil{\frac{d}{2}}}{2}+\sum_{i=2}^s (\floor{\frac{d-1}{2}}(n_i-1))\leq e(\textsf{G}_{n,d,\floor{\frac{d-1}{2}}})$ as $\sum_{i=1}^s n_i=n+s-1$ and $\floor{\frac{d-2}{2}}\left(n_1-\ceil{\frac{d}{2}}\right)+\binom{\ceil{\frac{d}{2}}}{2}\leq \floor{\frac{d-1}{2}}(n_1-\ceil{\frac{d+1}{2}})+\binom{\ceil{\frac{d+1}{2}}}{2}$.

If (ii) holds, then by (\ref{eqcase1}) we have $e(G)\leq \sum_{i=1}^{\ell} (\floor{\frac{d-1}{2}}(n_i-1)-1)+\sum_{i=\ell+1}^s(\floor{\frac{d-1}{2}}(n_i-1))\leq e(\textsf{G}_{n,d,\floor{\frac{d-1}{2}}})$.

If (iii) holds, then by the same calculation as in (ii) we have $e(G)=\sum_{i\in I}e(G_i)+\sum_{j\in [s]\setminus I}e(G_j)\leq \sum_{i\in I} (\floor{\frac{d-1}{2}}(n_i-1)-1)+\sum_{i\in [s]\setminus I}(\floor{\frac{d-1}{2}}(n_i-1))\leq e(\textsf{G}_{n,d,\floor{\frac{d-1}{2}}})$.

If (i), (ii) and (iii) do not hold, then we have $n_1<a$ and $n_{\ell}\leq d-2$ and the number of $G_i$ with $\ell \leq i\leq s$ which are not isomorphic to $K_{d-2}$ is at most $\ell-1$, so there must exist two blocks $G_{j_1}$ and $G_{j_2}$ isomorphic to $K_{d-2}$, since $n\geq k^{dk}>(\ell-1)(a+(d-2))+2(d-2)$. 
Note that since $G$ is $B(k,d)$-free and $G$ is connected, we must have $2(d-2)-1\leq k-1$ if $k-d+1\leq d-3$, thus $d\leq \frac{k+4}{2}$. 
By Theorem~\ref{Thm:Woodall-cycle} and $e(\textsf{F}_{n,d-1,1})\leq \frac{(d-2)(n-1)}{2}$, we have 
\[
\begin{displaystyle}
e(G) \leq e(\textsf{F}_{n,d-1,1})\leq 
\begin{cases}
e(\textsf{F}_{n,\frac{k+2}{2},1})\:\quad\quad\mbox{if  $d=\frac{k+4}{2}$ or $\frac{k+2}{2}$ }, \\
e(\textsf{F}_{n,\frac{k+3}{2},3})\:\quad\quad\mbox{if $d=\frac{k+3}{2}$}, \\
\floor{\frac{(k-d)n}{2}} \;\;\;\:\quad\quad \mbox{if $d\leq \frac{k+1}{2}$}.   
\end{cases}
 \end{displaystyle}
 \]

\smallskip

\noindent{\bf{Case 2.}} The circumference of $G$ is $d-1$.
By Theorem~\ref{thm:2-connected_Kopylov}, for each $i\in [s]$ and $n_i\geq d$ we have $e(G_i)\leq \max\{e(\textsf{G}_{n_i,d,2}), e(\textsf{G}_{n_i,d,\floor{\frac{d-1}{2}}})\}$.
Similarly, as in the previous case, we define threshold $a$. Here $a=\frac{5d-8}{4}$ if $d$ is even and $a=\frac{5d-11}{4}$ if $d$ is odd. Thus we have 
\begin{equation}\label{eqcase2}
\begin{displaystyle}
e(G_i) \leq
\begin{cases}
\floor{\frac{d-1}{2}}\left(n_i-\ceil{\frac{d+1}{2}}\right)+\binom{\ceil{\frac{d+1}{2}}}{2} \,\;\;\:\:\:\quad\mbox{if  $n_i\geq \max\{a, d\}$}, \\
\binom{d-2}{2}+2(n_i-(d-2))\:\quad\quad\quad\quad\quad\mbox{if $d\leq n_i< a$}, \\
\binom{n_i}{2} \;\:\quad\quad\quad\quad\quad\quad\quad\quad\quad\quad\quad\quad\quad \mbox{if $n_i\leq  d-1$}.
 \end{cases}
 \end{displaystyle}
\end{equation}

Let $G_i$ be a block that contains a cycle of order $d-1$. Then for every vertex $u\in V(G) \setminus V(G_i)$ we have $d(u)\leq k-d$ as $G$ is $B(k,d)$-free. 
Since $G_i$ is a maximal $2$-connected block, every vertex outside of $G_i$ is adjacent to at most one vertex of $G_i$, thus the number of edges from $G_i$ to $G-G_i$ is bounded by  $n-n_i$.
Thus if $d\geq \frac{k+4}{2}$ then by (\ref{eqcase2}) we have $e(G)\leq  e(G_i)+\frac{(k-d+1)(n-n_i)}{2}
\leq   e(\textsf{G}_{n,d,\floor{\frac{d-1}{2}}})$,  as $n$ is large.

If $d=\frac{k+3}{2}$ and $n_i\neq d-1$, then by (\ref{eqcase2}), $e(G)\leq  e(G_i)+\frac{(k-d+1)(n-n_i)}{2}
\leq   \max\{e(\textsf{G}_{n,d,\floor{\frac{d-1}{2}}}), e(\textsf{F}_{n,d,3})\}$.

If $d= \frac{k+3}{2}$ and $n_i=d-1$, let $G^\prime$ be a graph obtain from $G$ by contracting $G_i$ to a vertex. 
If $G^\prime$ is $\C_{\geq d-1}$-free, then by Theorem~\ref{Thm:Woodall-cycle} and (\ref{eqcase2}) we have $e(G)= e(G_i)+e(G')\leq e(G_i)+e(\textsf{F}_{n-n_i+1,d-1,1})\leq e(\textsf{F}_{n,d,2})$. 
Otherwise, $G$ contains two blocks of circumference $d-1$, then every vertex except at most one has degree at most $k-d$ and we have $e(G)\leq \frac{(k-d+1)(n-1)}{2}= e(\textsf{F}_{n,d,3})$ as $n$ is large.

If $d\leq \frac{k+2}{2}$, then by Theorem \ref{Thm:Woodall-cycle} we have 
$e(G)\leq e(\textsf{F}_{n,\frac{k+2}{2},1})$ if $d= \frac{k+2}{2}$ and $e(G)\leq \frac{(d-1)(n-1)}{2}\leq \floor{\frac{(k-d)n}{2}}$ if $d\leq \frac{k+1}{2}$.

This concludes the proof of Theorem~\ref{thm:brooms}.

\section{Concluding remarks}\label{sec:concluding remarks}
In this paper, for the Tur\'an problem, we determine the maximum effect the additional connectivity condition could have over all trees. An interesting future direction of research would be to identify appropriate parameters (if they exist) of a tree that determine the asymptotic behavior of its connected extremal number. The constructions in this paper could be useful for this problem. 

We remark that two natural parameters are relevant for the connected extremal number of a tree~$T$: (i) the maximum degree, and (ii) the size of the smaller color class of the bipartition of $T$. But these two parameters alone do not suffice to determine $\ex_c(n,T)$. Indeed, consider the binary tree $T$ with  $2^{2r}-1$ vertices. Note that  $\Delta(T)=3$ and the sizes of the bipartition classes of $T$ are  $\frac{2^{2r}-1}{3}$ and $2\cdot \frac{2^{2r}-1}{3}$. Using (i), one would take a $2$-regular graph, which contains only $n$ edges, and if we take (ii) into account and use the graph $\textsf G_{n,2\cdot\frac{2^{2r}-1}{3}-1,\frac{2^{2r}-1}{3}-1}$, we get a lower bound of at least $\approx (\frac{{2^{2r}-1}}{3}-1)n$ edges. 
However, consider a graph $\textsf{S}$ consisting of  $\floor{\frac{n-1}{2^{2r}-5}}$ vertex disjoint $K_{2^{2r}-7}$ with pendant path of length two, a path of $n-1-(2^{2r}-5)\floor{\frac{n-1}{2^{2r}-5}}$ vertices and a vertex $w$ adjacent to a vertex with degree one of each of these components. Clearly, $\textsf{S}$ has about $\frac{n(2^{2r}-1)}{2}$ edges. Note that  $\textsf{S}$ is $T$-free since every possible mapping of $T$ in $\textsf{S}$ must contain the vertex $w$ and a vertex of a clique,
contradicting the observation that the maximum length of the bare path in $T$ is $2$. Thus, perhaps the size of the maximum bare path in a tree $T$ is also relevant.

We refer the interested reader to the following papers, for problems and extensions related to Erd\H{o}s-S\'{o}s conjecture see \cite{Besomi2021boundedegree,rozhon2019local,goerlich2016erdHos,lidicky2012turan,mclennan2005erdHos,sacle1997withoutC4,wozniak1996erdos,Brandt1996girth5}, for extensions of Berge hypergraphs see \cite{gerbner2020general,gyHori2019tur}, for extensions of colored graphs see \cite{salia2019erdHos}. 

Caro, Patkós, and Tuza~\cite{caro2022connected} asked whether the connected extremal number becomes monotone eventually. In particular, for every graph $F$, there exists a constant $N_{F}$, such that for every $n\geq N_{F}$, $\ex_{c}(n,F)\leq \ex_{c}(n+1,F)$. We observe that this is true when $F$ contains a cycle. 

\begin{proposition}\label{prop:monotone}
For every graph $F$ containing a cycle, there exists a constant $N_F$ such that for every $n>N_F$ we have $\ex_{c}(n,F)<\ex_{c}(n+1,F)$.
\end{proposition}
\begin{proof}
    Let $n$ be sufficiently large and let $C_{\ell}$ be a cycle of $F$. Note that $\ex(n, C_{\ell})=\floor{\frac{n^2}{4}}$ if $\ell$ is odd and $\ex(n, C_{\ell})\geq cn^{1+4/(3\ell-4)}$ for some positive constant $c$ if $\ell$ is even and $n$ large enough (see \cite{Lazebnik1995new}).
    Thus, there exists $N_F$, such that if $n>N_F$ then $\ex(n,C_{\ell})>\frac{kn}{2}$, where $k$ is the number of vertices in $F$. 
    Note that the extremal graph for $C_\ell$ is connected, since $C_{\ell}$ is a $2$-connected graph.
    Thus we have $\ex_c(n,F) \geq \ex(n,C_{\ell})>\frac{kn}{2}$. 
    Therefore, by averaging, every connected extremal $n$-vertex $F$-free graph $G$ contains a vertex of degree at least $k$, for sufficiently large $n$. Let $G'$ be an $(n+1)$-vertex graph obtained from $G$, an $n$-vertex connected extremal graph of $F$, by adding a pendant vertex adjacent to a vertex of degree at least $k$, such vertex exists for every $n> N_F$. 
    It is easy to observe that the $(n+1)$-vertex connected graph  $G'$ is $F$-free and so $\ex_{c}(n,F)<\ex_{c}(n+1,F)$ as desired.
  \end{proof}

\textbf{Acknowledgement.} 
We would like to express our sincere appreciation to the anonymous referees for their valuable comments.

\bibliography{Proposal.bib}

\end{document}